\documentclass{article}

\usepackage{hyperref}

\usepackage[T1]{fontenc} 
\usepackage[utf8]{inputenc}
\usepackage[english=british]{csquotes}

\usepackage{amssymb}
\usepackage{amsmath} 
\usepackage{amsfonts}
\usepackage{mathrsfs}
\usepackage{amsthm}

\usepackage{tikz-cd}
\tikzcdset{
arrow style       = tikz,
cramped,
diagrams          = {>=stealth},
row sep/normal    = 2.2em,
column sep/normal = 2.7em,
}

\renewenvironment{tikzcd}{\begin{oldtikzcd}}{\end{oldtikzcd}\vspace{-3ex plus 3pt minus 1pt}}
\newcommand*{\ncd}{}

\newtheorem{theorem}{Theorem}[section]
\newtheorem{lemma}[theorem]{Lemma}
\newtheorem{proposition}[theorem]{Proposition}
\newtheorem{corollary}{Corollary}[theorem]
\newtheorem{remark}[corollary]{Remark}
\theoremstyle{definition}
\newtheorem{example}{Example}[theorem]
\newtheorem{definition}[theorem]{Definition}
\newtheorem{notation}[theorem]{Notation}

\newtheorem{situation}[subsection]{Situation}%

\numberwithin{equation}{section}

\usepackage[capitalise,noabbrev]{cleveref}

\newcommand*{\stacks}[1]{\href{http://stacks.math.columbia.edu/tag/#1}{Tag~#1}}

\newcommand*{\ulcorn}[1]{{\ar[#1, phantom, very near start, "\ulcorner" description]}}

\newcommand*{\classicsets}[1]{\mathbb{#1}} 

\newcommand*{\nat}{\classicsets{N}}   
\newcommand*{\rat}{\classicsets{Q}}   
\newcommand*{\aff}{\classicsets{A}}   
\newcommand*{\proj}{\classicsets{P}}  
\newcommand*{\field}{\Bbbk}           

\newcommand*{\hookrightstealtharrow}{%
  \mathrel{\begin{oldtikzcd}[ampersand replacement=\&, column sep=1.15em]\arrow[r,hook]\&{}\end{oldtikzcd}}}

\newcommand*{\rightstealtharrow}{%
  \mathrel{\begin{oldtikzcd}[ampersand replacement=\&, column sep=1.15em]\arrow[r]\&{}\end{oldtikzcd}}}

\newcommand*{\myapplication}[1]{\mspace{-8mu} #1 \mspace{-8mu}} 
\renewcommand*{\to}{\myapplication{\rightstealtharrow}}

\newcommand*{\tohook}{\myapplication{\hookrightstealtharrow}}

\newcommand*{\from}{{\colon\mspace{-3mu}}\linebreak[0]} 

\newcommand*{\fcat}[1]{\mathsf{#1}} 

\newcommand*{\set}{\fcat{Set}}
\newcommand*{\sch}{\fcat{Sch}}

\DeclareMathAlphabet{\mathpzc}{OT1}{pzc}{m}{it}
\newcommand*{\funct}[1]{\scalebox{1.2}{$\mathpzc{#1}$}} 
\newcommand*{\fiso}{\funct{Iso}}
\newcommand*{\fflat}{\mbox{\boldmath $\flat\kern-1.4pt\flat$}}

\DeclareMathOperator{\Cl}{Cl}
\DeclareMathOperator{\red}{red}
\DeclareMathOperator{\cns}{c}
\DeclareMathOperator{\h}{h}

\DeclareMathOperator{\Id}{Id}
\DeclareMathOperator{\bl}{bl}

\DeclareMathOperator{\image}{Im}
\newcommand*{\im}[1]{\image(#1)}

\DeclareMathOperator{\spec}{Spec}
\newcommand*{\s}[1]{\spec(#1)}

\newcounter{parnum}[section]
\renewcommand{\theparnum}{\thesection.\arabic{parnum}}
\renewcommand{\paragraph}{\medskip\refstepcounter{parnum}\textbf{\theparnum}~\textbf}

\title{The blow up split sections family}

\author{Laura Brustenga i Moncusí}

\begin{document}

\maketitle

\begin{abstract}
  The universal scheme of clusters of sections is an adaption of Kleiman's iterated blow ups (which parametrise clusters of points) to parametrise clusters of sections.
  They can also be constructed iteratively, but the iterative step is not so clear.
  Defining the blow up split sections family, we characterise this iterative step.
  Roughly speaking, it is a morphism that combines the universal properties of blow ups and universal section families.
  It is a generalisation of blow ups, and as such, we show that it exhibits some sort of birationality.
  But now, the flattening stratification of a morphism plays also an important role.
\end{abstract}

\section*{Introduction}

Let \(X\), \(Y\) be schemes over a ground scheme \(S\) and \(Z\) a closed subscheme of \(X_Y=X\times_S Y\).
Our main purpose is to introduce the blow up split section family (or the blow up \S family for short) of the projection \(X_Y\to Y\) along \(Z\).
It is a \(X\)-scheme \(\mathfrak{B}\overset{b}{\to} X\) such that the pullback of \(Z\) by \((b\times \Id_Y)\from \mathfrak{B}_Y\to X_Y\) is an effective Cartier divisor of \(\mathfrak{B}_Y\) and satisfying a suitable universal property.
Roughly speaking, it combines the universal properties of the universal section family, or Weil restriction, of \(X_Y\to Y\) (see~\cref{def:Weil-rest-and-Usf}) and of the blow up of \(X_Y\) along \(Z\).
Under Noetherian and projective assumptions, \cref{thr:blow-up-S-family-exists} asserts that
the blow up section family exists.

When \(Y\) is the base field we retrieve the classic blow up, but in general wide new phenomena may appear.
For example, the resulting morphism \(b\times \Id_Y\) is not necessarily birational or even generically finite, see \cref{ssec:examples}.
\medskip

Let \(f\from X_Y\to W\) be an \(S\)-morphism.
We also introduce the \(f\)-constfy closed subscheme of \(Y\), a fundamental step for the blow up \S family construction.
It follows from the study of morphisms \(T\to Y\) for which the restriction of \(f\) to \(X_T\) is constant along the fibres of the projection \(X_T\to T\).
\cref{thr:constantfying-closed} shows that they form a category with a final object, which is a closed subscheme of \(Y\).
It is the so called \(f\)-constfy closed subscheme of \(Y\) (see~\cref{def:f-constfy}).
The existence of the \(f\)-constfy closed subscheme of \(Y\) follows from the representability of the functor \(\fiso\) (see~\cref{sec:constfy,def:isofy-functor}).
The representability of this functor has been studied in the literature, but explicit constructions for the representing scheme are lacking.
We also introduce the class of \(\aleph_1\)-morphisms (see~\cref{def:Mittag-Leffler,def:aleph-1-homomorphism,def:aleph-1-morphisms}), which allow an explicit description for the representing scheme (see~\cref{thr:closed-isofying}).
\medskip

Our purpose introducing the blow up \S family is the following.
Fix a morphism \(\pi\from \mathcal{S}\to B\).
The author's paper \cite{brustenga-existeness-of-r-Urcf} introduces a generalisation of clusters of points of a scheme \(X\) to the relative case, clusters of sections of \(\pi\) (\cite[Definition 2.11, p.7]{brustenga-existeness-of-r-Urcf}).
There, the author adapts Kleiman's iterated blow ups (\cite[\S 4.~p.36]{kleiman-multiple-point-1981}, which naturally parametrise clusters of points) to parametrise clusters of sections of \(\pi\) of length \(r\), which led to the Universal \(r\)-relative cluster family \(\Cl_r\) of \(\pi\) (\cite[Definitions 2.13, 2.17 and 2.19, pp.8-10]{brustenga-existeness-of-r-Urcf}).

Assuming \(\mathcal{S}\) quasiprojective and \(B\) projective, first, the scheme \(\Cl_r\) is realised as a locally closed subscheme of a suitable Hilbert scheme, which proves its existence (see~\cite[Theorem 2.24, p.11]{brustenga-existeness-of-r-Urcf}).
The new \(f\)-constfy construction allow to relax the hypothesis to assume just that \(B\) is proper.
Second, it is shown that, as in \cite{kleiman-multiple-point-1981}, a recursive construction of \(\Cl_{r+1}\) from \(\Cl_r\) is possible.
But now, in general, the iterative step is more complex than a simple blow up.
The blow up \S family is our attempt to formalise and study such an iterative step.

More precisely, there is a stratification of \(\Cl_r\times_{\Cl_{r-1}}\Cl_r\) such that every irreducible component of \(\Cl_{r+1}\) is either (a) birational to the closure of a stratum or (b) composed entirely of clusters whose \((r+1)\)-th section is infinitely near to the \(r\)-th, see \cite[\S 2]{brustenga-existeness-of-r-Urcf} and \cite[Corollary 3.10.2]{brustenga-existeness-of-r-Urcf}.
So, each type (a) irreducible component is a blow up of the closure of a stratum along a suitable sheaf of ideals.
The blow up \S family is the morphism from the union of all type (a) irreducible components (with its non-necessarily reduced structure) to the whole scheme \(\Cl_r\times_{\Cl_{r-1}}\Cl_r\).
That is, it incorporates the stratification of \(\Cl_r\times_{\Cl_{r-1}}\Cl_r\) and strata-wise it is the corresponding blow up (see~\cref{thr:uni-ordinary-2-rcf-as-sbsf}).
\medskip

\cref{sec:prelim} introduces the basic constructions, and the notation, widely used in the forthcoming sections.

\cref{sec:blow-up-along-a-locally-principal}, we formalise the idea that blowing up a locally Noetherian scheme along a locally principal subscheme consists into shaving off those associated point of the ambient scheme lying on the locally principal subscheme.
We also show that, assuming \(Y\to S\) flat and with geometrically integral fibres, there is a one-to-one correspondence between the associated points of \(X\) and those of its base change \(X_Y\).
This all yields that, in this case, the blow up of \(X_Y\) along any locally principal subscheme is again the Cartesian product over \(S\) of \(Y\) with a closed subscheme of \(X\) (see~\cref{thr:blowup-locally-ppal-product-form}).

\cref{sec:constfy} presents the functor \(\fiso\) and its known existence theorems.
It also introduce the new class of morphism, the \(\aleph_1\)-morphisms, and presents the explicit construction of representing scheme.
Finally, there is the \(f\)-constfy construction.

\cref{sec:cross-blow-up} contains the construction of the blow up \S family, which is mainly based on the \(f\)-constfy and the Cartesian product form the blow up of \(X_Y\) along a locally principal subscheme.
It also contains a generalisation for blow up \S families of the fact that a blow up is an isomorphism away of its centre.

\cref{ssec:universal-relative-clusters-family} presents a way to apply the newly developed techniques for the construction of \(\Cl_{r+1}\) from \(\Cl_r\times_{\Cl_{r-1}}\Cl_r\) for the case \(r=1\), which is the inductive step for the whole construction.
The new techniques developed in this paper allow us also to describe, for now set theoretically, where type (b) irreducible components of \(\Cl_{r+1}\) emerge from, the ones missing in blow up \S family.

Finally \cref{ssec:examples} presents some particular examples of the blow \S family illustrating the new phenomena (the resulting morphism is not necessarily birational) and to show some possible applications such as to systematise small resolutions.
\medskip

{\textbf{Acknowledgements:}} I thank Prof. Joaquim Roé for his great support.

\section{Preliminaries}\label{sec:prelim}
This section introduces the basic constructions, and the notation, widely used in the forthcoming sections.
\medskip

Let \(X,Y\) be schemes.
Given a point \(x\) of \(X\), we denote by \(\kappa(x)\) its residue field and by \(\{x\}\) the scheme \(\s{\kappa(x)}\).
Usually we denote a monomorphism by \(Y\tohook X\) (almost all of them will be open or closed embeddings).


\paragraph{Category Theory.}
\label{ssec:category-theory}
We present a basic construction on category theory, which expresses the representability of a functor in terms of a universal property.

\begin{definition}\label{def:cat-of-objects-of-a-functor}
  Let \(\funct{F}\from \fcat{C}\to \set\) be a contravariant functor on a category \(\fcat{C}\) with values in sets.
  The \emph{category of elements of \(\funct{F}\)}, denoted by \(\int\funct{F}\), is the category whose objects are couples \((C,\eta)\) with \(C\) an object of \(\fcat{C}\) and \(\eta\) an element of \(\funct{F}(C)\).
  And arrows \((C,\eta)\to (C',\eta')\) in \(\int\funct{F}\) are arrows \(f\in\fcat{C}(C,C')\) such that \(\funct{F}(f)(\eta')=\eta\).

  Equivalently, the category \(\int \funct{F}\) may be defined as the comma category \((\h\downarrow \int\funct{F})\) or the opposite category to the comma category \((\mathbf{1}\downarrow \funct{F})\), where \(\h\from \fcat{C}\to \set\) is the Yoneda embedding and \(\mathbf{1}\from \fcat{C}\to \set\) is the constant functor with image the terminal object of \(\set\) (see~\cite[Chapter III]{mac-lane-categories} or \cite[Exercise 1.3.vi, p.22 and \S 2.4, pp.66--72]{riehl-2017-category-in-context}).
\end{definition}

\begin{remark}\label{rmk:from-rep-a-functor-to-final-object}
  An object \((C,\eta)\) of \(\int \funct{F}\) is terminal if and only if the couple \((C,\eta)\) represents \(\mathcal{F}\).
\end{remark}

\begin{lemma}\label{lem:iso-iff-majorize}
  Let \(\fcat{C}\) be a category and consider the following Cartesian square in \(\fcat{C}\),
  \[\begin{tikzcd}
      Z_{X} \rar{g} \dar[swap, hook]{j} \ulcorn{dr} & Z \dar[swap, hook]{i} \\
      X \rar{f} & Y \\
    \end{tikzcd}
  \] where \(i\) (hence also \(j\)) is a monomorphism.
  Then, \(j\) is an isomorphism if and only if there is an arrow \(h\from X\to Z\) such that \(i\circ h=f\).
\end{lemma}
\begin{proof}
  When \(j\) is an isomorphism, \(h=g\circ j^{-1}\).
  If there is such an arrow \(h\from X\to Z\), then \(j\circ (\Id_X\times_Y h)=\Id_X\) by definition, that is \(\Id_X\times_Yh\) is a section of (the monomorphism) \(j\).
  Hence, \(j\) is an isomorphism.
\end{proof}

\paragraph{Scheme theoretic image.}
\label{ssec:schematic-image-of-a-morphism}
We review the scheme theoretic image of a morphism, while we set the notation.

\begin{remark}\label{rmk:majorize-is-local-on-the-target}
  Fix a scheme \(Y\) and a monomorphism \(i\from Z\tohook Y\) (e.g. a closed or open embedding).
  Since isomorphisms are local in the target, by \cref{lem:iso-iff-majorize}, for a morphism \(f\from X\to Y\), the property of factorising through \(i\) is local on the source.
\end{remark}

\begin{definition}\label{def:sch-image}
  Let \(f\from X \to Y\) be a morphism of schemes.
  The \emph{scheme theoretic image of \(f\)} (or schematic image for short) is a closed subscheme \(\im{f}\) of \(Y\) through which \(f\) factorises and satisfying the following universal property: If \(f\) factorises through a closed embedding \(Z\tohook Y\), then \(\im{f} \tohook Y\) also factorises through it (see~\cite[Proposition 10.30]{ulrich-torsten-ag-i}, \cite[I Chapitre I, \S 9.5, p.176]{EGA-all} or \cite[\stacks{01R5}]{stacks-project}).
  We also call a diagram \(X \to \im{f} \tohook Y\) a scheme theoretic image.
  Given an open subscheme \(U\) of \(X\) the \emph{schematic closure of \(U\) in \(X\)} is the schematic image of the open embedding \(U\tohook X\).

  In addition, given a point \(x\) of \(X\), we denote by \(\overline{\{x\}}\) the schematic image of the natural morphism \(\s{\kappa(x)}\to X\).
\end{definition}

\begin{remark}\label{rmk:existence-of-sch-image}
  If \(f\) is quasi-compact, then the closed subscheme \(\im{f}\) of \(Y\) is defined by the quasi-coherent \(\mathscr{O}_Y\)-ideal \(\ker\!(\mathscr{O}_Y\to f_{*}\mathscr{O}_X)\).
\end{remark}

\begin{lemma}\label{lem:iso-descent-by-sch-image}
  Let \(X\to \overline{X}\tohook Y\) be a schematic image and \(i\from Z\tohook Y\) a closed subscheme.
  Then, the closed embedding \(Z_X\tohook X\) is an isomorphism if and only if so is \(Z_{\overline{X}}\tohook \overline{X}\).
\end{lemma}
\begin{proof}
  The closed embedding \(Z_X\tohook X\) is the base change of \(Z_{\overline{X}}\tohook \overline{X}\) by \(X\to \overline{X}\), hence if the latter is an isomorphism then so is the former.
  On the other side, if \(Z_X\tohook X\) is an isomorphism, via its inverse, the morphism \(X\to Y\) factorises through \(Z\tohook Y\).
  Then, by its universal property, the closed embedding \(\overline{X}\tohook Y\) also factorises through \(Z\tohook Y\) and the claim follows from \cref{lem:iso-iff-majorize}.
\end{proof}

The following is a standard result about schematic images (see~\cite[Lemma 14.6, p.424]{ulrich-torsten-ag-i}, \cite[\stacks{081I}]{stacks-project} or~\cite[IV$_2$ Chapitre IV, Proposition 2.3.2, p.14]{EGA-all}).

\begin{lemma}\label{lem:sch-image-commutes-with-flat-bc}
  Let \(S\) be a ground scheme and \(S' \to S\) a flat morphism.
  Let \(f\from X \to Y\) be a quasi-compact morphism of \(S\)-schemes with \(\overline{X}\) its schematic image.
  The schematic image of the base change \(f'\from X' \to Y'\) of \(f\) by \(S'\to S\) is the Cartesian product \(\overline{X}\times_SS'\).
\end{lemma}

\paragraph{Constant morphisms.}
\label{ssec:constant-morphisms}
We review the scheme theoretic version of a constant morphism.

\begin{definition}\label{def:constant-along-the-fibres-2}
  Let \(S\) be a ground scheme.
  Let \(p\from X\to Y\) and \(f\from X\to W\) be \(S\)-morphisms.
  Consider the following Cartesian diagram
  \[\begin{tikzcd}
      Z \rar{} \dar[hook]{} \ulcorn{dr} & W \dar[hook]{\Delta_{W/S}}\\
      X\times_Y X \rar{f\times_Y f} & W\times_S W \\
    \end{tikzcd}
  \] where \(\Delta_{W/S}\) is the diagonal.
  We say that the morphism \(f\) is \emph{constant along the fibres} of \(p\) if the monomorphism \(Z\tohook X\times_Y X\) is an isomorphism.
\end{definition}

The standard (and maybe more intuitive) definition of a morphism \(f\from X\to W\) being constant along the fibres of another morphism \(p\from X\to Y\) is that the following diagram commutes.
\[\begin{tikzcd}
    X\times_Y X \rar{q_1} \dar[swap]{q_2} & X \dar{f} \\
    X \rar[swap]{f} & W \\
  \end{tikzcd}
\]\ncd That is, the kernel, or equaliser, of the two morphisms \(f\circ q_1,f\circ q_2\) is the whole scheme \(X\times_Y X\), which, by \cref{lem:iso-iff-majorize}, is equivalent to \cref{def:constant-along-the-fibres-2} (see~\cite[Définition 1.4.2, p.34 and Proposition 1.4.10, p.37]{ega-i-2nd}).

\begin{remark}\label{rmk:const-along-stable-by-postcomposition}
  From the second definition follows straightforwardly that, given an \(S\)-morphism \(f'\from W\to W'\), if \(f\) is constant along the fibres of \(p\), then so is \(f'\circ f\).
  If furthermore \(f'\) is a monomorphism, then the converse also holds.
\end{remark}


\begin{proposition}\label{prop:const-along-fibres-and-rel-const}
  Let \(S\) be a ground scheme.
  Let \(p\from X\to Y\) and \(f\from X\to W\) be \(S\)-morphisms.
  If \(p\) is an fpqc morphism (see~\cite[Chapter 2, Definition 2.34, p.28]{fga-explained}, then \(f\) is constant along the fibres of \(p\) if and only if there is an \(S\)-morphism \(g\from Y\to W\) such that \(f=g\circ p\).
  In this case, the morphism is unique.
\end{proposition}
\begin{proof}
  It is a particular case of a bigger result on descent.
  Namely, the functor of points \(\h_{Y/S}\) of \(Y\to S\) is a sheaf in the fpqc topology (see \cite[Chapter 2, Theorem 2.55, p.34]{fga-explained}).
  The original result, due to Alexander Grothendieck, is \cite[B.1~Théorème 2.~(190-19)]{grothendieck-sb-i}, which applies to a slightly less general class of morphisms.
  The result may also be found at \cite[\stacks{03O3}]{stacks-project}.

  In this case, since there is just one element covering the whole scheme \(Y\), there is just one overlap, the scheme \(X\times_YX\) (such overlap would be trivial by the Zariski topology, but here it is not).
  So, whenever \(f\from X\to Y\) agrees with itself on this overlap, it extends uniquely to an \(S\)-morphism \(g\from Y\to W\).
  But this condition is equivalent to \(f\) being constant along the fibres of \(p\).
\end{proof}


\paragraph{Weil restrictions and families of sections.}
\label{ssec:families-sections}
We review two equivalent constructions, the Weil restriction and the universal sections family, and we state their main existence theorem.

\begin{definition}\label{def:Weil-rest-and-Usf}
Let \(S\) be a ground scheme and \(\pi\from X\to Y\) an \(S\)-morphism.
Consider the functor \(\funct{Sect}_{Y/S}(X)\from \sch_S\to \set\) sending an \(S\)-scheme \(T\to S\) to the set
\[
  \funct{Sect}_{Y/S}(X)(T)=\sch_Y(Y_T,X).
\] When this functor is representable, the representing scheme is called the \emph{Universal section family of \(\pi\)}, see \cite[II, C, nº2, pp.380,381, le foncteur ``ensemble des sections'']{grothendieck-sb-ii}.
\end{definition}

In the literature the Universal section family of \(\pi\) is studied from two different points of view.
It is also called the \emph{Weil restriction of \(\pi\)}, see \cite[\S 7.6, p.191]{bosch-1990-neron-models} and there are two main cases where the representability of the functor \(\funct{Sect}_{Y/S}(X)\) it is established.
\medskip

\cref{thr:Usf-representable} below is due to Alexander Grothendieck, see \cite[\S 4.c, pp.267,268]{grothendieck-sb-iv}.
For an alternative equivalent exposition see \cite{nitsure-construction-2005}.

\begin{theorem}\label{thr:Usf-representable}
Let \(S\) be a locally Noetherian ground scheme and \(\pi\from X\to Y\) an \(S\)-morphism.
If \(Y\to S\) is proper and flat and \(X\) is quasiprojective over \(S\), then \(\funct{Sect}_{Y/S}(X)\) is representable by a locally Noetherian quasiprojective \(S\)-scheme.
\end{theorem}

\begin{remark}\label{rmk:Usf-representable-for-picewise-qp}
  This result can be easily generalised to \(X\) piecewise quasiprojective, see \cite[Theorem 1.7, p.4]{brustenga-existeness-of-r-Urcf}.
\end{remark}

\cref{thr:Weil-restriction-representable} below can be found in \cite[Theorem 4, p.194]{bosch-1990-neron-models}.

\begin{theorem}\label{thr:Weil-restriction-representable}
Let \(S\) be a ground scheme and \(\pi\from X\to Y\) an \(S\)-morphism.
If \(Y\to S\) is finite and locally free and, for every point \(s\) of \(S\), every finite set \(P\) of points on the fibre \(X_s\) of \(X\to S\) is contained in an affine open subscheme of \(X\), then \(\funct{Sect}_{Y/S}(X)\) is representable by a locally Noetherian quasiprojective \(S\)-scheme.
\end{theorem}

\cref{lem:qprojective-affiness-finite-set-points} below is well-known (e.g., \cite[Proposition 3.36 (b), p.109]{liu-2002-AG-Arithmetic-curves}).

\begin{lemma}\label{lem:qprojective-affiness-finite-set-points}
  Let \(X\) be quasiprojective scheme over a ring \(A\).
  Then, every finite set \(P\) of points on \(X\) is contained in some affine open subscheme of \(X\).
\end{lemma}

Let \(\mathfrak{X}\) be a \(S\)-scheme and \((\psi\from \mathfrak{X}_Y\to X)\in \funct{Sect}_{Y/S}(X)(\mathfrak{X})\).
By \cref{rmk:from-rep-a-functor-to-final-object}, if the couple \((\mathfrak{X},\psi)\) represents the functor \(\funct{Sect}_{Y/S}(X)\), then it satisfies the following universal property:
For every \(S\)-scheme \(T\) and every \(\sigma\in \funct{Sect}_{Y/S}(X)(T)\), there is a unique \(S\)-morphism \(f\from T \to \mathfrak{X}\) such that the following diagram commutes.
\begin{equation}
  \begin{tikzcd}[row sep=tiny]
    T_Y \drar[]{\sigma} \ar[dd,swap,"f_Y"] & \\
    & X\\
    \mathfrak{X}_Y \urar[swap]{\psi} & \\
    \end{tikzcd}\label{dig:Usf}
\end{equation}\ncd

\section{Blowing up along a locally principal subscheme.}
\label{sec:blow-up-along-a-locally-principal}
Let \(S\) be a ground scheme and \(X\to S\) an \(S\)-scheme.
We show that blowing up a locally Noetherian scheme \(X\) along a locally principal subscheme \(Z\) consists of shaving off those associated points of \(X\) lying on \(Z\), \cref{thr:2n-loc-ppal-blowup}.
Given a flat \(S\)-scheme \(Y\to S\) with geometrically integral fibres, we show that there is a one-to-one correspondence, preserving specialisations, between the associated points of \(X\) and those of \(X\times_SY\), \cref{lem:associated-pt-of-product-over-flat-base}.
This all yields that, the blow up of \(X\times_SY\) along any locally principal subscheme is again the Cartesian product over \(S\) of \(Y\) with a closed subscheme of \(X\), see~\cref{thr:blowup-locally-ppal-product-form}.
\medskip

Let \(X\) be a scheme.
We recall that a \emph{locally principal subscheme} of \(X\) is a closed subscheme whose sheaf of ideals is locally generated by a single element, whereas an \emph{effective Cartier divisor} of \(X\) is a closed subscheme whose sheaf of ideals is locally generated by a single \emph{regular} element (see~\cite[Remark 6.17.1, p.145]{hartshorne-algebraic-1977}, \cite[Definition 11.24, p.301]{ulrich-torsten-ag-i}, \cite[\stacks{01WQ}]{stacks-project} or \cite[IV$_4$ Chapitre IV, Définition 21.1.6, p.257, and Paragraphe 21.2.12, p.262]{EGA-all}).
\medskip

Let \(f,g\from X\to Y\) be two morphisms and \(U\) an open subscheme of \(X\).
When \(U\) is (topologically) dense in \(X\), the equation \(f|_U=g|_V\) implies \(f|_{X_{\red}}=g|_{X_{\red}}\) but not generally \(f=g\).
That motivates the following definition.

\begin{definition}\label{def:schematic-dense-open-subsch}
  Let \(X\) be a scheme.
  An open subscheme \(U\) of \(X\) is \emph{scheme theoretically dense} in \(X\) if, for every open \(V\) of \(X\), the schematic closure of \(U\cap V\) in \(V\) is equal to \(V\) (see~\cite[\stacks{01RB}]{stacks-project} or \cite[IV$_3$ Chapitre IV, Définition 11.10.2, p.171]{EGA-all}).
\end{definition}

\begin{remark}\label{rmk:sch-dense-to-dense}
  In general, there are schemes \(X\) with open subschemes \(U\) which are not schematically dense although \(\overline{U}=X\) (see~\cite[\stacks{01RC}]{stacks-project}).
  But, when the ambient scheme \(X\) is locally Noetherian, every open embedding is quasicompact (see~\cite[\stacks{01OX}]{stacks-project} or \cite[I Chapitre I, Proposition 6.6.4, p.153]{EGA-all}) and then an open subscheme \(U\tohook X\) is schematically dense if and only if \(\overline{U}=X\) (see~\cite[\stacks{01RD}]{stacks-project} or \cite[IV$_3$ Chapitre IV, Remarque 11.10.3 (iv), p.171]{EGA-all}).
\end{remark}

\begin{proposition}\label{prop:1r-loc-ppal-blowup}
  Let \(X\) be a scheme and \(Z\) a closed subscheme of \(X\).
  Let \(i\from U \to X\) be the open subscheme complement of \(Z\) in \(X\) and \(b\from \overline{U} \tohook X\) its schematic closure.
  If \(Z\) is a locally principal subscheme of \(X\), then the closed embedding \(b\from \overline{U} \tohook X\) is the blow up of \(X\) along \(Z\).
\end{proposition}
Note that if \(Z\) is an effective Cartier divisor then \(\overline{U}=X\) (see~\cite[IV$_2$ Chapitre IV, Corollaire 3.1.9, p.38]{EGA-all}).
\begin{proof}[sketch]
  We may assume \(X\) affine, say \(X\cong \s{A}\) for some ring \(A\), and \(Z\) defined by a principal ideal, say \((f)\subseteq A\).
  Then, the open subscheme \(U\) of \(X\) is \(D(f)\cong \s{A_f}\) and the closed embedding \(b\) is given by the natural homomorphism \(A \to A/\mathfrak{a}\) where \(\mathfrak{a}=\ker(A \to A_f)\subseteq A\).
  Furthermore, we may assume \(f\in A\) non-nilpotent, otherwise the result is trivial.
  Then, \(\mathfrak{a}=\cup_{n\in \nat}(0:f^n)\) is a proper ideal of \(A\) and it satisfies the following universal property: Ever homomorphism \(\varphi\from A \to B\) such that \(\varphi(f)\in B\) is a non-zerodivisor, factorises through \(A\to A/\mathfrak{a}\).
  Hence, \(\overline{U}\) is the blow up of \(X\) along \(Z\).
\end{proof}

The blow up of any scheme \(X\) along any locally principal subscheme is just the schematic closure of its open complement.
But, when the scheme \(X\) is locally Noetherian, there are no pathological associated points, see \cite[\stacks{02OI}]{stacks-project}, and then, as \cref{thr:2n-loc-ppal-blowup} below shows, we can understand much better which parts of \(Z\) are shaved off on the blowing up procedure.

\begin{theorem}\label{thr:2n-loc-ppal-blowup}
  Let \(X\) be a locally Noetherian scheme and \(Z\) a locally principal subscheme of \(X\).
  Let \(T_Z\) be the subset of \(X\) union of the underlying sets of \(\overline{x}\) for all \(x\in Ass(X)\cap Z\).
  Let \(V\) be its complement in \(X\).
  Then \(V\) is an open subscheme of \(X\) and its schematic closure \(\overline{V}\tohook X\) is the blow up of \(X\) along \(Z\).
\end{theorem}
\begin{proof}
  First of all, the subset \(T_Z\) of \(X\) is closed because its intersection with every Notherian affine open subscheme of \(X\) is a union of finitely many closed subsets (see~\cite[\stacks{05AF}]{stacks-project} or \cite[IV$_2$ Chapitre IV, Proposition 3.1.6, p.37]{EGA-all}).
  Hence \(V\) is an open subscheme of \(X\).

  Let \(U\) be the open complement of \(Z\) and \(\overline{U}\) its schematic closure.
  Since \(T_Z\) is a closed subset of \(Z\), \(U\) is an open subscheme of \(V\) and of \(\overline{V}\).
  We show that \(U\tohook \overline{V}\) is schematically dense, then the claim follows from \cref{prop:1r-loc-ppal-blowup}.

  By definition of \(T\), \(Ass(X)\cap U = Ass(X)\cap V\) and, by \cite[IV$_2$ Chapitre IV, Proposition 3.1.13, p.39]{EGA-all}, \(Ass(\overline{V})\subseteq Ass(X)\cap V\).
  So, \(Ass(\overline{V})\subseteq U\) and then \(U\) is a schematically dense subscheme of \(\overline{V}\) (see~\cite[IV$_3$ Chapitre IV, Proposition 11.10.10, p.172]{EGA-all}).
\end{proof}



\begin{lemma}\label{lem:associated-pt-of-product-over-flat-base}
  Let \(S\) be a locally Noetherian ground scheme.
  Let \(X\overset{f}{\to} S\) and \(Y\overset{g}{\to} S\) be locally Noetherian \(S\)-schemes.
  Let \(\eta\in Ass(X)\), set \(s=f(\eta)\in S\) and consider the following Cartesian diagram.
  \[\begin{tikzcd}
      (Y_s)_{\eta} \rar{} \dar{} \ulcorn{dr} & \{\eta\} \dar{} \\
      Y_s \rar{} & \{s\}\\
    \end{tikzcd}
  \]\ncd Assume that \(g\) is flat and with geometrically integral fibres.
  Then, the scheme \((Y_s)_{\eta}\) is integral and its generic point is mapped to an associated point \(\xi_{\eta}\) of \(X\times_SY\) by the natural monomorphism \((Y_s)_{\eta}\tohook X\times_SY\).
  Furthermore, the map sending \(\eta\in Ass(X)\) to \(\xi_{\eta}\in Ass(X\times_SY)\) is a one-to-one correspondence, which preserves specialisations.
\end{lemma}
\begin{proof}
  The scheme \(Y_s\) is integral because we assume \(g\) with geometrically integral fibres.
  Denote the generic point of \(Y_s\) by \(\mu\) and denote by \(I_{\eta}\) the image of \(Ass(\s{\kappa(\eta)\otimes_{\kappa(s)}\kappa(\mu)})\) by the natural monomorphism \(\s{\kappa(\eta)\otimes_{\kappa(s)}\kappa(\mu)}\tohook X\times_SY\).
  By \cite[IV$_2$, Chapter IV, Proposition 3.3.6, p.44]{EGA-all},
  \[
    Ass(X\times_SY) = \bigcup_{\eta\in Ass(X)}I_{\eta}.
  \] Observe that the natural monomorphism \(\s{\kappa(\eta)\otimes_{\kappa(s)}\kappa(\mu)}\tohook X\times_SY\) factorises as
  \[
    \s{\kappa(\eta)\otimes_{\kappa(s)}\kappa(\mu)}\tohook (Y_s)_{\eta}\tohook X\times_SY
  \] Moreover, again by \cite[IV$_2$, Chapter IV, Proposition 3.3.6, p.44 or Corollaire 3.3.7, p.45]{EGA-all}, the associated points of \(\s{\kappa(\eta)\otimes_{\kappa(s)}\kappa(\mu)}\) are mapped to associated points of \((Y_s)_{\eta}\), which is integral because we assume \(g\) with geometrically integral fibres.
  Hence, there is a unique point in \(Ass(\s{\kappa(\eta)\otimes_{\kappa(s)}\kappa(\mu)})\) and \(\xi_{\eta}\in Ass(X\times_SY)\) is its image to \(X\times_SY\).
  \medskip

  Now, we sketch the proof that the map sending \(\eta\in Ass(X)\) to \(\xi_{\eta}\in Ass(X\times_SY)\) preserves specialisations.
  Fix another \(\eta'\in Ass(X)\), set \(s'=f(\eta')\) and assume that \(\eta\) is a specialisation of \(\eta'\), that is \(\eta\in \overline{\{\eta'\}}\).
  By transitivity of schematic images (see \cite[I, Chapitre I, Proposition 9.5.5, p.177]{EGA-all}), there is a morphism \(\overline{\{\eta'\}}\to \overline{\{s'\}}\) such that
  \begin{equation}\label{eq:ext-ass-pt-respects-specialisations}
    Y\times_S \overline{\{\eta'\}} = \Bigl(Y\times_S \overline{\{s'\}}\Bigr)\times_{\overline{\{s'\}}} \overline{\{\eta'\}}.
  \end{equation} From \cref{eq:ext-ass-pt-respects-specialisations} and \cref{lem:sch-image-commutes-with-flat-bc} follows that the schematic image of the generic point of \((Y_{s'})_{\eta'}\) in \(X\times_SY\) is \(Y\times_S \overline{\{\eta'\}}\), which is also the schematic image of \(\{\xi_{\eta'}\}\).
  Now, by \cref{eq:ext-ass-pt-respects-specialisations} and the morphisms \(\{\eta\}\to \overline{\{\eta'\}}\) and \(\{s\}\to \overline{\{s'\}}\), there is a morphism \((Y_s)_{\eta}\to Y\times_S \overline{\{\eta'\}}\), which implies that \(\xi_{\eta}\in \overline{\{\xi_{\eta'}\}}\).
\end{proof}

\begin{remark}\label{rmk:projecting-xi-eta}
  Recall that the image of \(\xi_{\eta}\) by the projection \(X\times_SY\to X\) is \(\eta\).
\end{remark}

\begin{lemma}\label{lem:product-form-1}
  Let \(S\) be a ground scheme.
  Let \(Y\to S\) an fpqc morphism.
  Let \(X\) be an \(S\)-scheme and \(i\from W\tohook X\) a closed embedding.
  Let \(h'\from T\to X\) be an \(S\)-morphisms.
  Let \(\varphi\from T\times_S Y\to W\times_SY\) be a morphism such that the following diagram commutes.
  \[\begin{tikzcd}
      T\times_S Y \rar{\varphi} \drar[swap]{h'_Y} & W\times_SY \dar[hook]{i_Y} \\
      & X\times_SY \\
    \end{tikzcd}
  \]\ncd
  Then, there is a unique morphism \(h\from T\to W\) such that \(\varphi=h_Y\).
\end{lemma}
\begin{proof}
  Denote by \(p_T\from T\times_SY\to T\), \(p_X\from X\times_SY\to X\) and \(p_W\from W\times_SY\to W\) the projections.
  Since the following diagram commutes,
  \[\begin{tikzcd}
      T\times_SY \dar{p_T} \rar{i_Y\circ\varphi} & X\times_SY \rar{p_X} & X \\
 T \ar[urr,swap,bend right=10,"h'"] \\
\end{tikzcd}
\]\ncd the morphism \(p_X\circ i_Y\circ\varphi\) is constant along the fibres of \(p_T\).
Then, since \(p_X\circ i_Y=i\circ p_W\) and \(i\) is a monomorphism, by \cref{rmk:const-along-stable-by-postcomposition}, the morphism \(p_W\circ\varphi\) is constant along the fibres of \(p_T\).
By \cref{prop:const-along-fibres-and-rel-const}, there is a unique morphism \(h\from T\to W\) such that \(h\circ p_T=p_W\circ\varphi\).
Consider the following diagram.
\[\begin{tikzcd}
    T_Y \rar{\varphi} \dar{p_T} & W_Y \rar{} \dar{p_W} \ulcorn{dr} & Y \dar{} \\
    T \rar{h} & W \rar{} & S \\
\end{tikzcd}
\]\ncd Since it commutes and both the right hand and the big squares are Cartesian, so is the left hand.
Hence, \(\varphi= h_X\).
\end{proof}

\begin{theorem}\label{thr:blowup-locally-ppal-product-form}
  Let \(S\) be a locally Noetherian ground scheme.
  Let \(X\overset{f}{\to} S\) and \(Y\overset{g}{\to} S\) be locally Noetherian \(S\)-schemes.
  Let \(Z\) be a locally principal subscheme of \(X\times_SY\).
  Assume that \(Y\overset{g}{\to} S\) is flat and with geometrically integral fibres.
  Then, there is a closed subscheme \(i\from W\to X\) such that the closed embedding \(i_Y\from W\times_SY\to X\times_SY\) is the blow up of \(X\times_SY\) along \(Z\).

  If furthermore \(Y\to S\) is an fpqc morphism, for every \(S\)-scheme \(T\overset{h'}{\to} S\) for which the preimage of \(Z\) by \(h'_X\from T\times_SY \to X\times_SY\) is an effective Cartier divisor, there is a unique morphism \(h\from T\to W\) such that \(i\circ h'= h\).
  Moreover, \(h_X\from T\times_SY\to W\times_SY\) is the morphism given the universal property of the blow up \(i_Y\).
\end{theorem}
\begin{proof}
  Let \(\Omega\) denote the set of points \(\xi\in Ass(X\times_SY)\) such that \(\xi\in Z\).
  By \cref{thr:2n-loc-ppal-blowup}, the blow up of \(X\times_S Y\) along \(Z\) is the schematic closure of the open subscheme \(U\tohook X\times_SY\) complement of the closed subset
  \[
    T_Z=\bigcup_{\xi\in\Omega}\overline{\{\xi\}}.
  \]
  Let \(p\from X\times_SY\to X\) be the projection and denote by \(V\) the open subscheme of \(X\) complement of the closed subset
  \[
    \bigcup_{\xi\in\Omega} \overline{\{p(\xi)\}}.
  \] We claim that the schematic closure of the open embedding \(V\tohook X\) is the desired closed subscheme \(W\) of \(X\).
  Let us check it.
  Observe that, since \(g\) is assumed flat, by \cref{lem:sch-image-commutes-with-flat-bc}, the schematic closure of the open embedding \(V\times_SY\tohook X\times_SY\) is \(\overline{V}\times_SY\).
  An associated point \(\eta\) of \(U\) is a point \(\eta\in Ass(X\times_SY)\) such that \(\eta\not\in \overline{\{\xi\}}\) for all \(\xi\in\Omega\).
  Since the one-to-one correspondence between \(Ass(X)\) and \(Ass(X\times_SY)\) respects specialisations, this is equivalent to \(p(\eta)\not\in\overline{\{p(\xi)\}}\) for all \(\xi\in\Omega\), which is equivalent to \(\eta\in p^{-1}(V)=V\times_SY\).
  Hence, \(Ass(U)=Ass(V\times_SY)\).
  Since \(\xi\in p^{-1}(p(\xi))\), the scheme \(V\times_SY\) is an open subscheme of \(U\) and then the schematic closures of \(U\) and \(V\times_SY\) in \(X\times_SY\) are equal (see \cite[IV$_2$, Chapitre IV, Proposition 3.1.13, p.39 and IV$_3$, Chapitre IV, Proposition 11.10.10, p.172]{EGA-all} or \cite[\stacks{083P}]{stacks-project}).
  \medskip

  Assume that \(Y\to S\) is an fpqc morphism and consider such an \(S\)-scheme \(T\).
  By the universal property of the blow up \(i_Y\), there is a unique morphism \(\varphi\from T\times_SY\to W\times_SY\) such that \(i_Y\circ \varphi=h'_X\).
  Now, the claim follows from \cref{lem:product-form-1}.
\end{proof}

\begin{remark}\label{rmk:geometrically-integral-fibres-required}
  If the assumption \(Y\to S\) with geometrically integral fibres fails, then there is a point \(s\) of \(S\) and a field extension \(\kappa(s)\tohook K\) such that \((Y_s)_K\) is not integral.
  Setting \(X= \s{K}\), the scheme \(X\times_SY\) is \((Y_s)_K\) and it has at least one locally principal subscheme \(Z\), which is not an effective Cartier divisor.
  Hence, the blow up of \(X\times_SY\) along \(Z\) is not an isomorphism and, if it is not the empty scheme (otherwise \cref{thr:blowup-locally-ppal-product-form} is trivial), there is no closed subscheme \(W\) of \(X\) such that \(W\times_SY\tohook X\times_SY\) is such a blow up.
\end{remark}

\section{The constfy closed subscheme}
\label{sec:constfy}

Let \(S\) be a ground scheme.
Let \(p\from X\to Y\) and \(f\from X \to W\) be \(S\)-morphisms.
We study \(S\)-morphisms \(T\to Y\) for which the restriction of \(f\) to \(X_T\) is constant along the fibres of the projection \(X_T\to T\).
\cref{thr:constantfying-closed}, an immediate consequence of \cref{thr:closed-isofying}, shows that they form a category with a final object, which is a closed subscheme of \(Y\).
We called it the \(f\)-constfy closed subscheme of \(Y\) (see~\cref{def:f-constfy}).

To study this category, we use the functor \(\fiso\) (see~\cref{def:isofy-functor}).
The representability of this functor has been studied in the literature, but explicit constructions for the representing scheme are lacking.
So, we also introduce the class of \(\aleph_1\)-morphisms (see~\cref{def:Mittag-Leffler,def:aleph-1-homomorphism,def:aleph-1-morphisms}), which allow an explicit description for the representing scheme (see~\cref{thr:closed-isofying}).

\begin{definition}\label{def:Mittag-Leffler}
  Let \(R\) be a ring.
  An \(R\)-module \(M\) is \emph{Mittag-Leffler} if the natural homomorphism
\[
  \rho\from M\otimes_R\prod_{i\in I}Q_i\to \prod_{i\in I}M\otimes_RQ_i
\] is injective for every family of \(R\)-modules \((Q_i\,|\,i\in I)\).
\end{definition}

We are interested in Mittag-Leffler modules which moreover are flat.
In \cite{herbera-trlifaj-2012-M--L-conditions}, there is a complete characterisation of such modules as \(\aleph_1\)-projective modules, which motivates the following definition (see \cite[Corollary 2.7, p.3443 and Corollary 2.10, p.3444]{herbera-trlifaj-2012-M--L-conditions}).

\begin{definition}\label{def:aleph-1-homomorphism}
  We say that an homomorphism \(\varphi\from A\to B\) is \emph{\(\aleph_1\)-projective} if \(B\) is a flat and Mittag-Leffler \(A\)-module via \(\varphi\).
\end{definition}

\begin{lemma}\label{lem:extending-ideals-to-nice-tensor-product}
  Let \(A\to B\) be an \(\aleph_1\)-projective homomorphism.
  Then, for every family of ideals \(\{\mathfrak{a}_{\lambda}\}_{\lambda\in\Lambda}\) of \(A\),
  \[
    B\cdot\bigcap_{\lambda\in\Lambda}\mathfrak{a}_\lambda =\bigcap_{\lambda\in\Lambda}B\cdot\mathfrak{a}_{\lambda}.
  \]
\end{lemma}
\begin{proof}
Since \(B\) is a flat \(A\)-module, the following sequence is exact.
\[\begin{tikzcd}
    0 \rar[hook]{} & B\otimes_A\bigcap_{\lambda\in \Lambda} \mathfrak{a}_{\lambda} \rar[hook]{} & B\otimes_AA \rar{\alpha} & B\otimes_A\displaystyle\prod_{\lambda\in\Lambda}A/\mathfrak{a}_{\lambda} \\
\end{tikzcd}
\]\ncd So, \(B\cdot\cap_{\lambda}\mathfrak{a}_{\lambda}=\ker(\alpha)\).
Now, since \(B\) is a Mittag-Leffler \(A\)-module, the natural homomorphism
\[
  \rho\from B\otimes_A\prod_{\lambda\in\Lambda}A/\mathfrak{a}_{\lambda}\to \prod_{\lambda\in\Lambda} B\otimes_AA/\mathfrak{a}_{\lambda}
\] is injective.
Hence, \(\ker(\alpha)=\ker(\rho\circ\alpha)=\cap_{\lambda}B\cdot\mathfrak{a}_{\lambda}\).
\end{proof}

\begin{definition}\label{def:aleph-1-morphisms}
  Let \(f\from X\to Y\) be morphism.
  An \emph{\(\aleph_1\)-projective covering of \(f\)} is a couple \((\mathcal{U},\mathcal{V})\) where \(\mathcal{U}=\{U_i\}_i\) is an affine open cover of \(Y\) and \(\mathcal{V}=\{V_{i,j}\}_{i,j}\) is a collection of  affine open covers \(\{V_{i,j}\}_j\) of \(f^{-1}(U_i)\) for ever \(i\), such that for every \(i,j\) the homomorphism corresponding to \(V_{i,j}\to U_i\) is \(\aleph_1\)-projective.
  We say that \(f\from X\to Y\) is \emph{\(\aleph_1\)-projective},
  if it admits an \(\aleph_1\)-projective covering.
\end{definition}

\begin{example}\label{rmk:example-1-of-aleph-1-morphisms}
  Let \(\field\) be a field.
  Let \(X\), \(Y\) be \(\field\)-schemes, then the projection \(X\times_{\field} Y\to X\) is \(\aleph_1\)-projective.
  Fix affine covers \(\mathcal{U}=\{U_i\}\), \(\{V_j\}\) of \(X\), \(Y\) respectively.
  Then, the set \(\mathcal{V}=\{U_i\times V_j\}\) is an affine cover of \(X\times Y\) and the couple \((\mathcal{U},\mathcal{V})\) is an \(\aleph_1\)-projective covering of \(X\times Y\to X\).
  Let us check it.

  For every \(i,j\), the projection \(U_i\times V_j\to U_i\) corresponds to the natural homomorphism \(A\to A\otimes_{\field}B\) for some \(\field\)-algebras \(A\), \(B\).
  So, \(A\otimes_{\field}B\) is a free \(A\)-module and free modules are flat (well-known) and Mittag-Leffler (see~\cite[\stacks{059Q}]{stacks-project}).
\end{example}

\begin{example}\label{rmk:example-2-of-aleph-1-morphisms}
  For the same reason, an affine morphism \(f\from Z\to S\) such that the \(\mathscr{O}_S\)-module \(f_{*}\mathscr{O}_Z\) is locally free is \(\aleph_1\)-projective, and its pullbacks by a morphism of this same type is again \(\aleph_1\)-projective.
\end{example}

\begin{notation}\label{not:schematic-unions}
  Let \(X\) be a scheme.
  Consider a family of quasi-coherent \(\mathscr{O}_X\)-ideals  \(\{\mathscr{I}_l\}_l\) and its corresponding to a closed subschemes \(Y_l\)  of \(X\).
  We denote its schematic union by \(\Sigma_lY_l\).
  More precisely, the scheme \(\Sigma_lY_l\) is the closed subscheme of \(X\) corresponding to the quasi-coherent \(\mathscr{O}_X\)-ideal \(\bigcap_l\mathscr{I}_l\).
\end{notation}

\cref{prop:extending-arbitrary-families-of-cl-ssch} below is the main property for which we introduce \(\aleph_1\)-projective morphisms.
It asserts that arbitrary schematic unions of closed subscheme commute with \(\aleph_1\)-projective pullbacks.

\begin{proposition}\label{prop:extending-arbitrary-families-of-cl-ssch}
  Let \(X\to Y\) be an \(\aleph_1\)-projective morphism.
  Then, for every family \(\{Y_l\}_l\) of closed subscheme of \(Y\), the closed subschemes \(X_{\Sigma_l Y_l}\) and \(\Sigma_lX_{Y_l}\) of \(X\) are equal.
\end{proposition}
\begin{proof}
  Fix an \(\aleph_1\)-projective covering \((\{U_i\},\{V_{i,j}\})\) of \(X\to Y\).
  We check that for every \(i,j\) the closed subschemes \((X_{\Sigma_l Y_l})\cap V_{i,j}\) and \((\Sigma_lX_{Y_l})\cap V_{i,j}\) of \(V_{i,j}\) are equal.

  Fix \(i,j\) and denote respectively by \(A\) and \(B\) the rings of functions of \(U_i\) and \(V_{i,j}\).
  Every closed subscheme \(Y_l\cap U_i\) of \(U_i\) is given by an ideal \(\mathfrak{a}_l\) of \(A\).
  The closed subschemes \((X_{\Sigma_l Y_l})\cap V_{i,j}\) and \((\Sigma_lX_{Y_l})\cap V_{i,j}\) of \(V_{i,j}\) are given respectively by the ideals \(\cap _lB\cdot \mathfrak{a}_l\) and \(B\cdot \cap_l\mathfrak{a}_l\).
  But since \(B\) is an \(\aleph_1\)-projective \(A\)-module by assumption, by \cref{lem:extending-ideals-to-nice-tensor-product}, such ideals are equal.
\end{proof}

\begin{definition}\label{def:isofy-functor}
  Let \(p\from X\to Y\) and \(Z\to X\) be morphisms.
  We define \(\fiso_p^Z \from \sch_Y\to \set\) as the contravariant functor sending an \(Y\)-scheme \(T\to Y\) to
  \[
    \fiso_p^Z(T)=
    \begin{cases}
      \{*\} & \mbox{if }Z_T\to X_T \mbox{ is an isomorphism},\\
      \emptyset & \mbox{otherwise}.
    \end{cases}
 \] Since isomorphisms are stable by base change, it is well defined over morphisms.
\end{definition}

\begin{remark}\label{rmk:underlying-set-cl-ssch-isofy}
  If the functor \(\fiso_p^Z\) is representable by an open or closed subscheme \(Y'\) of \(Y\), the underlying set of \(Y'\) is
  \[
    \omega=\{y\in Y \mbox{ such that }Z_y\tohook X_y \mbox{ is an isomorphism}\}.
  \] If a point \(y\) of \(Y\) belongs to \(Y'\), then \(Z_y\tohook X_y\) is the base change of (the isomorphism) \(Z_{Y'}\tohook X_{Y'}\) by \(y\to Y'\), hence \(y\in \omega\).
  If \(y\in\omega\), then, by the universal property of the closed embedding \(Y'\tohook Y\), the morphism \(\{y\}\to Y\) factorises through \(Y'\tohook Y\).
  Hence, \(y\) belongs to \(Y'\).
\end{remark}

There are two main different cases when the representablility of the functor \(\fiso_p^Z\) has been studied.
We state them for the convenience of the reader.
\medskip

The following can be found in \cite[\stacks{07AI}]{stacks-project}.

\begin{theorem}\label{thr:isofy-representable-by-closed}
  Let \(p\from X\to Y\) be a morphism and \(Z\tohook X\) a closed embedding.
  If \(p\) is of finite presentation, flat, and pure, then \(\fiso_f^Z\) is representable and the representing scheme \(Y'\) is a closed subscheme of \(Y\).
  Moreover, if \(Z\to Y\) is of finite presentation, then so is \(Y'\tohook Y\).
\end{theorem}

\cref{thr:isofy-representable-by-open} below, by \cref{rmk:from-rep-a-functor-to-final-object}, is equivalent to \cite[Chapter 5, Theorem 5.22 (b), p.132]{fga-explained}.

\begin{theorem}\label{thr:isofy-representable-by-open}
  Let \(p\from X\to Y\) and \(Z\to X\) be morphisms.
  If \(Y\) is Noetherian, \(Z\to X\) is projective and \(Z\), \(X\) are proper and flat over \(Y\), then \(\fiso_p^Z\) is representable in the category of locally Noetherian \(Y\)-schemes and the representing scheme \(Y'\) is an open subscheme of \(Y\).
\end{theorem}

\begin{remark}\label{rmk:isofy-representable-by-connected-component}
  Notice that a proper morphism onto a Noetherian scheme is of finite presentation (trivially) and pure (see \cite[\stacks{05K3}]{stacks-project}).
  Hence, if furthermore \(Z\to X\) is a closed embedding, by \cref{thr:isofy-representable-by-closed}, the scheme \(Y'\) representing \(\fiso_p^Z\) is a connected component of \(Y\).
\end{remark}

\begin{theorem}\label{thr:closed-isofying}
  Let \(p\from X\to Y\) be a morphism and \(Z\) a closed subscheme of \(X\).
  Let \(\Omega\) denote the set of closed subschemes \(W\) of \(Y\) such that \(Z_W\tohook X_W\) is an isomorphism and denote by \(\Sigma_{\Omega}\) the closed subscheme \(\Sigma_{W\in \Omega}W\) of \(Y\).
  If \(p\) is \(\aleph_1\)-projective, then the scheme \(\Sigma_{\Omega}\) represents the functor \(\fiso_{p}^{Z}\).
\end{theorem}

By \cref{rmk:from-rep-a-functor-to-final-object}, a closed subscheme \(Y'\) of \(Y\) represents the functor \(\fiso_{p}^{Z}\) if and only if a morphism \(T\to Y\) factorises through \(Y'\tohook Y\) whenever the closed embedding \(Z_T\tohook X_T\) is an isomorphism.

\begin{proof}[of \cref{thr:closed-isofying}]
  For every \(W\in\Omega\), the isomorphism \(Z_W\tohook X_W\) is an \(X\)-morphism, hence the closed embeddings \(Z_W\tohook X\) and \(X_W\tohook X\) correspond to the same closed subscheme of \(X\).
  Then, the schemes \(\Sigma_{W\in\Omega}Z_W\) and \(\Sigma_{W\in\Omega}X_W\) are the same subscheme of \(X\) and, by   \cref{prop:extending-arbitrary-families-of-cl-ssch}, the closed embedding \(Z_{(\Sigma_{\Omega})}\tohook X_{(\Sigma_{\Omega})}\) is an isomorphism, in fact an \(X\)-isomorphism.
  So, if a morphism \(T\to Y\) factorises through \(\Sigma_{\Omega}\), the closed embedding \(Z_T\tohook X_T\) is an isomorphism.
  \medskip

  Now, given a morphism \(T\to Y\) such that the closed embedding \(Z_T\tohook X_T\) is an isomorphism, by \cref{lem:iso-descent-by-sch-image}, the schematic image \(\overline{T}\) of \(T\to Y\) is a closed subscheme of \(Y\) belonging to \(\Omega\).
  Hence, there is a unique \(Y\)-morphism \(\overline{T}\tohook \Sigma_{\Omega}\) and then, by composition, there is a unique \(Y\)-morphism \(T\to \Sigma_{\Omega}\).
\end{proof}

\begin{definition}\label{def:f-constfy}
  Let \(S\) be a ground scheme.
  Let \(p\from X\to Y\) and \(f\from X \to W\) be \(S\)-morphisms.
  Let \(Y'\) be a closed subscheme of \(Y\).
  We call \(Y'\) a \emph{\(f\)-constfy} closed subscheme of \(Y\), if the morphism \(f|_{X_{Y'}}\from X_{Y'}\to W\) is constant along the fibres of the projection \(X_{Y'}\to Y'\) and it satisfies the following universal property:
  A morphism \(T\to Y\) factorises through \(Y'\tohook Y\) if and only if \(f|_{X_{T}}\) is constant along the fibres of the projection \(X_{T}\to T\).
\end{definition}

If a \(f\)-constfy closed subscheme exists, by abstract nonsense it is uniquely determined up to a unique isomorphism.

\begin{theorem}\label{thr:constantfying-closed}
  Let \(S\) be a ground scheme.
  Let \(p\from X\to Y\) and \(f\from X \to W\) be \(S\)-morphisms.
  If \(W\) is separated over \(S\) and \(p\) is flat and proper, then the \(f\)-constfy closed subscheme of \(Y\) exists.
\end{theorem}
\begin{proof}
  Consider the following Cartesian diagram.
  \[\begin{tikzcd}
      Z \rar{} \dar[hook]{} \ulcorn{dr} & W \dar[hook]{\Delta_{W/S}} \\
      X\times_YX \rar{f\times_Yf} & W\times_S W  \\
\end{tikzcd}
\]\ncd Since \(W\) is separated, \(Z\tohook X\times_YX\) is a closed embedding and, since \(p\) is flat and proper, so is \(g\from X\times_YX\to Y\).
Hence, by \cref{thr:isofy-representable-by-closed}, the functor \(\fiso_g^Z\) is represented by a closed subscheme \(Y'\) of \(Y\).
We claim that \(Y'\) is the \(f\)-constfy closed subscheme of \(Y\).
\medskip

It is straightforward to check that the following diagram is Cartesian.
\[\begin{tikzcd}
    X_{Y'}\times_{Y'}X_{Y'} \rar[hook]{} \dar{} & X\times_Y X \dar{g} \\
    Y'  \rar[hook]{} & Y\\
\end{tikzcd}
\]\ncd So, since \(Z_{Y'}\tohook (X\times_YX)\times_YY'\) is an isomorphism, \(f|_{X_{Y'}}\) is constant along the fibres of the projection \(X_{Y'}\to Y'\).
Furthermore, now it is clear that \(Y'\) satisfies the required universal property.
\end{proof}

\begin{remark}\label{rmk:commutes-of-iso-and-constfy}
  Let \(S\) be a ground scheme.
  Let \(p\from X\to Y\) and \(f\from X \to W\) be \(S\)-morphisms.
  Let \(Z\tohook X\) be a closed subscheme of \(X\).
  In this situation, we may iterate the constructions of the \(f\)-constfy closed subscheme of \(Y\) and the closed subscheme of \(Y\) representing the functor \(\fiso_p^Z\).
  Assuming existence, it is straightforward to see that both possible ways of iterating such constructions give the same closed subscheme of \(Y\).
\end{remark}

\section{The blow up \S family}
\label{sec:cross-blow-up}

Consider the following situation.

\begin{situation}\label{sit:blow-up-section-family}
  Let \(S\) be a ground scheme.
  Let \(X\), \(Y\) be \(S\)-schemes with \(Y\to S\) an fpqc morphism.
  Consider the scheme \(X_Y=X\times_SY\) and denote by \(\pi\from X_Y\to Y\) and \(\alpha\from X_Y\to X\) the projections.
  Let \(Z\) be a closed subscheme of \(X_Y\).
  \[
    \begin{tikzcd}
      Z \rar[hook]{\mathbf{cl.emb.}} & X_Y \dar[swap]{\pi} \rar{\alpha} \ulcorn{dr} & X \dar{} \\
    &  Y \rar{\mathbf{fpqc}} & S  \\
\end{tikzcd}
\]\ncd
\end{situation}

In this section we prove our main result, \cref{thr:blow-up-S-family-exists}, which asserts the existence of the blow up \S family of the projection \(X_Y\to Y\) along \(Z\) (see~\cref{def:blowup-sections-family}) under suitable assumptions.
The blow up \S family is a generalisation of blow ups, as such \cref{thr:structure-of-blowup-sf-Cartesian-case} is the corresponding generalisation of the well-known fact that a blow up is an isomorphism away of its centre.

\begin{definition}\label{def:blowup-sections-family}
  Consider \cref{sit:blow-up-section-family}.
  Let \(\mathfrak{B}\) be an \(S\)-scheme and \(b\from \mathfrak{B}\to X\) an \(S\)-morphism.
  \[
    \begin{tikzcd}
      (b_Y)^{-1}(Z) \rar[hook]{} \dar{} \ulcorn{dr} & \mathfrak{B}_Y \dar[swap]{b_Y} \rar{} \ulcorn{dr} & \mathfrak{B} \dar{} \\
      Z \rar[hook]{} & X_Y \rar{\alpha} & X \\
    \end{tikzcd}
  \]\ncd
  We call the couple \((\mathfrak{B},b)\) a \emph{blow up split section family of \(\pi\) along \(Z\)} (or blow up \S family for short) if \((b_Y)^{-1}(Z)\tohook \mathfrak{B}_Y\) is an effective Cartier divisor and it satisfies the following universal property:
  For every \(S\)-morphism \(g\from T\to X\) for which \((g_Y)^{-1}(Z)\tohook T_Y\) is an effective Cartier divisor, there is a unique morphism \(h\from T \to \mathfrak{B}\) such that \(b\circ h=g\).
  Analogously to classic blow ups, we call \(Z\) the \emph{centre} of the blow up \S family and \(b^{-1}(X_Y)\) the \emph{exceptional divisor} in \(\mathfrak{B}_Y\).
\end{definition}

If a blow up \S family exists, by abstract nonsense it is uniquely determined up to a unique isomorphism.

\begin{theorem}\label{thr:blow-up-S-family-exists}
  Consider \cref{sit:blow-up-section-family} assuming all the schemes locally Noetherian.
  If \(X_Y\) is piecewise quasiprojective over \(S\), \(X\) is separated over \(S\) and \(Y\to S\) is a morphism with geometrically integral fibres and it is finite locally free or proper and flat, then the blow up \S family of \(\pi\) along \(Z\) exists.
\end{theorem}
\begin{proof}
  Consider the blow up \(\bl\from \bl(Z,X_Y)\to X_Y\) of \(X_Y\) along \(Z\).
  The scheme \(\bl(Z,X_Y)\) is again piecewise quasiprojective over \(S\), then, by \cref{thr:Usf-representable} or \cref{thr:Weil-restriction-representable}, the universal section family \((\mathfrak{X},\psi)\) of \((\pi\circ \bl)\from \bl(Z,X_Y)\to Y\) exists.
  So now, we may consider the following diagram,
  \[\begin{tikzcd}
      \mathfrak{X}_Y \dar{} \rar{\psi} & \bl(Z,X_Y) \rar{\bl} & X_Y \rar{\alpha} & X \\
      \mathfrak{X} \\
    \end{tikzcd}
  \]\ncd where \(\mathfrak{X}_Y\to \mathfrak{X}\) is the projection.
  By \cref{thr:constantfying-closed}, the \((\alpha\circ \bl\circ \psi) \)-constfy closed subscheme \(\mathfrak{Z}\) of \(\mathfrak{X}\) exists.
  Denote by \(i\from \mathfrak{Z}\tohook \mathfrak{X}\) its corresponding closed embedding.
  Now, by construction the morphism \(\alpha\circ \bl\circ \psi\circ i_Y\) is constant along the fibres of the projection \(p\from \mathfrak{Z}_Y\to \mathfrak{Z}\), hence, by \cref{prop:const-along-fibres-and-rel-const}, there is a morphism \(v\from \mathfrak{Z}\to X\) such that \(v\circ p= \alpha\circ \bl\circ\psi\circ i_Y\).
  Consider the following diagram.
  \[\begin{tikzcd}
      \mathfrak{Z}_Y \dar{p} \rar{\bl\circ \psi\circ i_Y} & X_Y \rar{\pi} \dar{\alpha} & Y \dar{} \\
      \mathfrak{Z} \rar{v} & X \rar{} & S \\
    \end{tikzcd}
  \]\ncd Since it commutes and both the right hand and the big squares are Cartesian, so is the left hand.
  That is, \(\bl\circ \psi\circ i_Y= v_Y\).
  Finally, since \((v_Y)^{-1}(Z)\) is the preimage by \(\psi\circ i_Y\) of the exceptional divisor in \(\bl(Z,X_Y)\), it is locally principal and, by \cref{thr:blowup-locally-ppal-product-form}, there is a closed subscheme \(\mathfrak{B}\) of \(\mathfrak{Z}\) such that the closed embedding \(\mathfrak{B}_Y\tohook \mathfrak{Z}_Y\) is the blow up of \(\mathfrak{Z}_Y\) along \((v_Y)^{-1}(Z)\).
  Denote by \(b\from \mathfrak{B}\to X\) the restriction of \(v\) to \(\mathfrak{B}\).
  \medskip

  Now, it is straightforward to check that the couple \((\mathfrak{B},b)\) is the blow up \S family of \(\pi\) along \(Z\).
  It follows by applying iteratively the universal properties of the objects used to construct \(\mathfrak{B}\) and, at the last step, \cref{thr:blowup-locally-ppal-product-form}.
\end{proof}

Consider \cref{sit:blow-up-section-family} assuming \(X\) connected, \(Y\) integral, Noetherian and projective and flat over \(S\).
\cref{thr:structure-of-blowup-sf-Cartesian-case} below is the generalisation to blow up \S families to the well-know fact that a blow up is an isomorphism away of its centre.

\begin{notation}\label{not:struture-blowup-sf}
  We recall that the so called ``flattening stratification'' of the morphism \(Z\to X\) is a finite stratification
  \[
    X = \sqcup_{\Phi\in\rat[t]}X_{\Phi}
  \] by locally closed subschemes such that for every \(\Phi\), the pullback of \(Z\to X\) by \(X_{\Phi}\tohook X\) is flat and the Hilbert polynomial of the fibres is constant equal to \(\Phi\), and moreover, a morphism \(T\to X\) factorises through \(\sqcup_{\Phi}X_{\Phi}\tohook X\) if and only if the projection \(Z_T\to T\) is flat (see~\cite[Chapter 5, Theorem 5.13, p.123 and \S 5.5.6, universal property \textbf{(F)}, p.129]{fga-explained} or \cite[Lemma 2.3 (flattening), p.64]{kleiman-1980-compac-picar-schem}).

  Since \(X_Y\to X\) is flat, the Hilbert polynomial of its fibres is constant, say \(\Phi_0\).
  By \cref{thr:isofy-representable-by-closed}, the functor \(\fiso_{X_Y\to X}^Z\) is representable by a closed subscheme \(X_0\) of \(X\).
  Observe that, by \cref{rmk:underlying-set-cl-ssch-isofy}, the underlying sets of \(X_{\Phi_0}\) and \(X_0\) are equal.
  In fact, it is not hard to see that they are the same closed subscheme of \(X\), but we will not use it.

  By \cite[Chapter 9, Lemma 9.3.4, p.258]{fga-explained}, for every \(\Phi\), the points \(x\in X_{\Phi}\) for which \(Z_x\tohook (X_Y)_x=Y_x\) is an effective Cartier divisor form a (possibly empty) open subscheme of \(X_{\Phi}\), we denote it by \(U_{\Phi}\).
\end{notation}

\begin{remark}\label{rmk:smooth+2flat-all/none-fibres-Cartier}
  By \cite[Théorème 2.1 (i), p.231]{grothendieck-sb-vi}, if furthermore \(Y\) is smooth over \(S\), the open subscheme \(U_{\Phi}\) of \(X_{\Phi}\) is also a closed subset.
  Hence, \(U_{\Phi}\) is either the empty scheme or a connected component of \(X_{\Phi}\).
\end{remark}

\begin{definition}\label{def:core-of-blowup-Sfamilies}
Consider \cref{sit:blow-up-section-family} assuming \(X\) connected, \(Y\) integral, Noetherian and projective and flat over \(S\).
Consider also \cref{not:struture-blowup-sf}.
We call the scheme \(X_0\) the \emph{core} of the blow up \S family of \(\pi\) along \(Z\).
\end{definition}


\begin{theorem}\label{thr:structure-of-blowup-sf-Cartesian-case}
  Consider \cref{sit:blow-up-section-family} assuming \(X\) connected, \(Y\) integral, Noetherian and projective and flat over \(S\).
  Consider also \cref{not:struture-blowup-sf}.
  Assume that the blow up \S family \((\mathfrak{B},b)\) of \(\pi\) along \(Z\) exists.
  Then, the open subscheme \(\mathfrak{B}\setminus b^{-1}(X_0)\) of \(\mathfrak{B}\) is isomorphic to \(\sqcup_{\Phi}U_{\Phi}\).
\end{theorem}
\begin{proof}
  Denote by \(E\) the exceptional divisor in \(\mathfrak{B}_Y\), that is \(E=(b_X)^{-1}(Z)\).
  Clearly, the closed subscheme \(b^{-1}(X_0)\) of \(\mathfrak{B}\) represents the functor \(\fiso_{\mathfrak{B}_Y\to \mathfrak{B}}^E\), hence \(\mathfrak{B}\setminus b^{-1}(X_0)\) is the set of points \(b\in \mathfrak{B}\) for which \(E_b\tohook X_b\) is not an isomorphism.
  Then, since \(E\tohook \mathfrak{B}_Y\) is an effective Cartier divisor and \(X\) is integral, \(\mathfrak{B}\setminus b^{-1}(X_0)\) is the open subset (by \cite[Chapter 9, Lemma 9.3.4, p.258]{fga-explained}) corresponding to the set of points \(b\in \mathfrak{B}\) for which \(E_b\tohook X_b\) is an effective Cartier divisor.
  Then, by \cite[\stacks{062Y}]{stacks-project}, \(E\cap (\mathfrak{B}\setminus b^{-1}(X_0))\to \mathfrak{B}\setminus b^{-1}(X_0)\) is flat and then, by the universal property of the flattening stratification, there is a unique morphism \(\mathfrak{B}\setminus b^{-1}(X_0)\to \sqcup_{\Phi}X_{\Phi}\) (whose image clearly is contained in \(\sqcup_{\Phi}U_{\Phi}\)) such that the corresponding diagram commutes.
  Hence, it factorises through \(\sqcup_{\Phi}U_{\Phi}\tohook \sqcup_{\Phi}X_{\Phi}\) via a unique morphism \(\xi\from (\mathfrak{B}\setminus b^{-1}(X_0))\to \sqcup_{\Phi}U_{\Phi}\).

  Now, by construction and again by \cite[\stacks{062Y}]{stacks-project}, \(Z_{\sqcup_{\Phi}U_{\Phi}}\tohook X_{\sqcup_{\Phi}U_{\Phi}}\) is an effective Cartier divisor, hence, by the universal property of \((\mathfrak{B},b)\), there is a unique morphism \(\sqcup_{\Phi}U_{\Phi}\to \mathfrak{B}\) (whose image is contained in \(\mathcal{B}\setminus b^{-1}(X_0)\) because \(U_{\Phi_0}\) is empty) such that the corresponding diagram commutes.
  So finally, \(\sqcup_{\Phi}U_{\Phi}\to \mathfrak{B}\) factorises through \(\mathfrak{B}\) via a unique morphism \(\varepsilon\from \sqcup_{\Phi}U_{\Phi}\to (\mathfrak{B}\setminus b^{-1}(X_0))\).

  Now, it is straightforward to check that \(\xi\) and \(\varepsilon\) are mutually inverse.
\end{proof}

\begin{corollary}\label{coro:when-T0-is-empty}
    Consider \cref{sit:blow-up-section-family} assuming \(Y\) integral, Noetherian and projective and flat over \(S\).
    If there are no point \(x\) of \(X\) such that the fibre \(Z_x\tohook Y_x\) is an isomorphism, then the blow up \S family of \(\pi\) along \(Z\)  exists and it is the natural morphism \(\sqcup_{\Phi}U_{\varphi}\tohook X\).
\end{corollary}
\begin{proof}
  In this case the core of the blow up \S family is empty.
\end{proof}

\section{Universal 2-relative clusters family}
\label{ssec:universal-relative-clusters-family}
Fix a morphism \(\pi\from \mathcal{S}\to B\) with \(\mathcal{S}\) piecewise quasiprojective and \(B\) projective and integral.
So, its universal section family \((X,\psi)\) exists (see~\cref{def:Weil-rest-and-Usf}).
Here, we present the final goal of this paper, namely the construction of the universal \((r+1)\)-relative cluster section family \(\Cl_{r+1}\) of \(\pi\) from \(\Cl_r\times_{\Cl_{r-1}}\Cl_r\).
The general construction requires introduce a lot of notation.
So, we restrict to the case \(r=1\), that is \(\Cl_2\) (see~\cref{def:2-Urcf}) from \(X\times X\), which is the inductive step for the whole construction.

The following is a preliminary proposition. We leave its proof to the reader.

\begin{proposition}\label{prop:Usf-of-S-0}
  Let \(\mathcal{S}_0\) be the scheme \(\mathcal{S}\times B\), \(\pi_0\from \mathcal{S}_0\to B\times X\) the morphism \(\pi\times \Id_X\) and \(p\from B\times X\to B\) the projection.
  Let \(\psi_0\from B\times X\times X\to \mathcal{S}_0\) be the morphism \((\psi\times \Id_X)\circ \iota\) where \(\iota\from B\times X\times X\to B\times X\times X\) is the automorphism that twists the second and third factors.
  Then the universal section family of \(\pi_0\) is \((X, \psi_0)\) and the universal section family of \(p\circ \pi_0\) is \((X\times X,\psi_0)\).
\end{proposition}

Now, fix the following notation.
The scheme \(\mathcal{S}_1\) is the blow up of \(\mathcal{S}_0\) along the image \(\Delta\) of the section \((\psi\times_X \Id_X)\from B\times X\to \mathcal{S}\times X\) of \(\pi_0\), the scheme \(E\) is the exceptional divisor in \(\mathcal{S}_1\) and the morphism \(\pi_1\from \mathcal{S}_1\to B\times X\) is the composition of the blow up morphism \(\mathcal{S}_1\to \mathcal{S}_0\) and \(\pi_0\).
In addition, denote by \(q_1\from \mathcal{S}_1\to X\) the composition of the blow up morphism and the projection \(\mathcal{S}_0\to X\).
By \cite[\S 2]{brustenga-existeness-of-r-Urcf} and \cite[Corollary 3.10.2]{brustenga-existeness-of-r-Urcf}, there is a stratification of \(X\times X\) such that every irreducible component of \(\Cl_2\) is either (a) birational to a stratum or (b) composed entirely of clusters whose second section is infinitely near to the first.
\cref{thr:uni-ordinary-2-rcf-as-sbsf} below asserts that the blow up \S family \((X',b)\) of the projection \(B\times X\times X\to B\) along \((\psi_0)^{-1}(\Delta)\) (see~\cref{def:blowup-sections-family}) is the union, with its non-necessarily reduced structure, of the kind (a) irreducible components of \(\Cl_2\).

To finish this section we show that the kind (b) irreducible components of \(\Cl_2\), the ones missing in the blow up \S family construction, may only emerge from the universal section family of \(E\to B\), see \cref{thr:factoring-through-both-X'-XEc}.

\begin{notation}\label{not:second-universal-sf-and-b'}
  We denote the universal section family of \((p\circ \pi_1)\from \mathcal{S}_1\to B\) by \((X_1,\psi_1)\) (see~\cref{def:Weil-rest-and-Usf}).

  We denote by \(b'\from B\times X'\to \mathcal{S}_1\) the unique morphism whose composition with the blow up morphism \(\mathcal{S}_1\to \mathcal{S}_0\) is equal to \(\Id_B\times b\).
\end{notation}

\begin{definition}\label{def:2-Urcf}
  A \emph{2-relative cluster family of \(\pi\)} is a section family \((W,\theta)\) of \((p\circ \pi_1)\from S_1\to B\) such that the morphism \(q_1\circ \theta\) is constant along the fibres of the projection \(B\times W\to W\).

  A \emph{universal 2-relative cluster family of \(\pi\)} is a 2-relative cluster family \((\Cl_2,\rho)\) of \(\pi\) that satisfies the following universal property.
  For every 2-relative cluster family \((W,\theta)\) of \(\pi\), there is a unique morphism \(f\from W\to \Cl_2\) such that \(\theta=\rho\circ (\Id_B\times f)\).
\end{definition}

Notice that \cref{def:2-Urcf} is simpler than but equivalent to \cite[Definition 2.19, p.10]{brustenga-existeness-of-r-Urcf}.
That is because here we use the existence of the universal section family of \(\pi\).
Recall that, by \cref{prop:const-along-fibres-and-rel-const}, for every 2-relative cluster family \((W,\theta)\) of \(\pi\), the morphism \(q_1\circ \theta\from B\times W\to X\) is constant along the fibres of the projection \(B\times W\to W\) if and only if there is a morphism \(W\to X\) such that \(q_1\circ \theta\) commutes with the composition \(B\times W\to W\to X\).

When a universal 2-relative cluster family of \(\pi\) exists, by abstract nonsense it is unique up to unique isomorphism.

\begin{theorem}\label{thr:2-Urcf-by-constfy}
  The \((q_1\circ\psi_1)\)-constfy closed subscheme \(X_1^{\cns}\) of \(X_1\) exists and, setting \(\psi^{\cns}=\psi_1|_{B\times X_1^{\cns}}\), the couple \((X_1^{\cns},\psi^{\cns})\) is the universal 2-relative cluster family of \(\pi\from \mathcal{S}\to B\).
\end{theorem}
\begin{proof}
  It follows immediately from the universal properties of the universal section family \((X_1,\psi_1)\) of \((p\circ \pi_1)\from \mathcal{S}_1\to B\) and of the \((q_1\circ\psi_1)\)-constfy closed subscheme \(X_1^{\cns}\) of \(X_1\).
\end{proof}

\begin{proposition}\label{prop:Usf-of-E-to-B}
  Let \(X_E\) be the closed subscheme of \(X_1\) representing the functor \(\fiso_{X_1/B}^{\psi_1^{-1}(E)}\) and set \(\psi_E=\psi_1|_{B\times X_E}\).
  Then, the couple \((X_E,\psi_E)\) is the universal section family of \(E\to B\).
\end{proposition}
\begin{proof}
  Clearly the couple \((X_E,\psi_E)\) is a section family of \(E\to B\), let us check that it satisfies the required universal property.

  Let \((Y,\rho)\) be a section family of \(E\to B\).
  By the universal property of \((X_1,\psi_1)\) there is a morphism \(f\from Y\to X_1\) such that \(\rho=\psi_1\circ (\Id_B\times f)\).
  Then, by the transitivity of the Cartesian product and \cref{lem:iso-iff-majorize}, the base change of \(\psi_1^{-1}(E)\tohook B\times X_1\) by \(f\from Y\to X_1\) is an isomorphism.
  So, \((f\from Y\to X_1)\in \fiso_{X_1/B}^{\psi_1^{-1}(E)}(Y)\) and there is a unique morphism \(g\from Y\to X_E\) whose composition with \(X_E\tohook X_1\) is \(f\).
  Now, using that \(E\tohook S_1\) is a monomorphism, it is straightforward to check that \(\rho=\psi_E\circ (\Id_B\times g)\).
\end{proof}

\begin{notation}\label{not:W'-of-2-rcf-W}
  Let \((W,\theta)\) be a 2-relative cluster family of \(\pi\).
  Let \(E_W\) be the pullback of \(E\tohook \mathcal{S}_1\) by \(\theta\) (which is a locally principal subscheme of \(B\times W\)).
  We denote by \(W'\) the closed subscheme of \(W\) for which the closed embedding \(B\times W'\tohook B\times W\) is the blow up of \(B\times W\) along \(E_W\) (see~\cref{thr:blowup-locally-ppal-product-form}).
\end{notation}

\begin{notation}\label{not:W_E-of-2-rcf-W}
  Let \((W,\theta)\) be a 2-relative cluster family of \(\pi\).
  We denote by \(W_E\) the closed subscheme of \(W\) representing the functor \(\fiso_{W/B}^{E_W}\) (see~\cref{thr:closed-isofying}).
\end{notation}

For the particular case \(W=X_1^{\cns}\), we simplify the notation \((X_1^{\cns})'\) by \(X_1^{\cns\prime}\) and \((X_1^{\cns})_E\) by \(X^{\cns}_E\).
In addition we denote respectively by \(\psi_1^{\cns}\) and \(\psi^{\cns}_E\) the restrictions of \(\psi^{\cns}\from B\times X_1^{\cns}\to \mathcal{S}_1\) to \(B\times X_1^{\cns\prime}\) and \(B\times X^{\cns}_E\).
\medskip

\begin{remark}\label{rmk:double-def-of-X_E-cns}
  The scheme \(X^{\cns}_E\) is isomorphic to the \((q_1\circ \psi_E\from B\times X_E\to X)\)-constfy closed subscheme of \(X_E\), see \cref{rmk:commutes-of-iso-and-constfy}.
\end{remark}

\begin{proposition}\label{prop:uni-ordinary-2-rel-cluster-f-as-bowup}
  The couple \((X_1^{\cns\prime},\psi_1^{\cns\prime})\) satisfies the following universal property.
  For all 2-relative cluster family \((W,\theta)\) of \(\pi\) such that the pullback \(\theta^{-1}(E)\) is an effective Cartier divisor of \(B\times W\), there is a unique morphism \(f\from W\to X_1^{\cns\prime}\) with \(\theta=\psi_1^{\cns\prime}\circ(\Id_B\times f)\).
\end{proposition}
\begin{proof}
  Let \((W,\theta)\) be such a 2-relative cluster family of \(\pi\).
  By the universal property of \((X_1^{\cns},\psi^{\cns})\) (see~\cref{thr:2-Urcf-by-constfy}), there is a unique morphism \(f\from W\to X_1^{\cns}\) with \(\theta=\psi^{\cns}\circ(\Id_B\times f)\).
  By the universal property of the blow up \(\Id_B\times i\from B\times X_1^{\cns\prime}\tohook B\times X_1^{\cns}\), there is a unique morphism \(\varphi\from B\times W\to B\times X_1^{\cns\prime}\) with \(\Id_B\times f=\varphi\circ (\Id_B\times i)\).
  Finally, \cref{lem:product-form-1} asserts that the morphism \(\varphi\) is the product of \(\Id_B\) with a unique morphism \(W\to X_1^{\cns\prime}\).
\end{proof}

\begin{proposition}\label{prop:foctoring-morphisms-g'-gE}
  Let \((W,\theta)\) be a 2-relative cluster family of \(\pi\).
  Then, there are unique morphisms \(g'\from W'\to X'\) and \(g_E\from W_E\to X_E^{\cns}\) such that \(\theta|_{B\times W'}=\psi_1^{\cns\prime}\circ (\Id_B\times g')\) and \(\theta|_{B\times W_E}=\psi^{\cns}_E\circ(\Id_B\times g_E)\).
\end{proposition}
\begin{proof}
  Let denote by \(i\) the closed embedding \(W'\tohook W\).
  Since the composition of \(\theta_{B\times W'}\from B\times W\to \mathcal{S}_1\) with the blow up \(\mathcal{S}_1\to \mathcal{S}_0\) is a section family of \(\pi_0\from \mathcal{S}_0\to B\), by \cref{prop:Usf-of-S-0}, there is a unique morphism \(h\from W\to X\times X\) such that the corresponding diagram (\ref{dig:Usf}) commutes.

  Since \(E_W\tohook B\times W\) is also the pullback of \(\psi_0^{-1}(\Delta)\tohook B\times X\times X\) by \(\Id_B\times h\),
  the pullback of \(\psi_0^{-1}(\Delta)\) by \(\Id_B\times (h\circ i)\) is an effective Cartier divisor of \(B\times W'\) and then such unique morphism \(g\from W\to X'\) exists by the universal property of \((X',b)\).

  The base change of \(E_W\tohook B\times W\) by \(W_E\tohook W\) is an isomorphism and, via its inverse, \(W_E\) is a section family of \(E\to B\).
  Hence, first by the universal property of \((X_E,\psi_E)\) and second by the universal property of the \((q_1\circ \psi_E)\)-constfy closed subscheme \(X_E^{\cns}\) of \(X_E\) (see~\cref{rmk:double-def-of-X_E-cns}), such morphism \(g_E\from W_E\to X_E^{\cns}\) exists.
\end{proof}

\begin{remark}\label{rmk:factoring-through-X'}
  Let \((W,\theta)\) be a 2-relative cluster family of \(\pi\).
  If \(E_W\tohook B\times W\) is an effective Cartier divisor, then \(W=W'\).
\end{remark}

\begin{theorem}\label{thr:uni-ordinary-2-rcf-as-sbsf}
  The couple \((X',b')\) satisfies the same universal property of \((X_1^{\cns\prime},\psi_1^{\cns\prime})\) (see~\cref{prop:uni-ordinary-2-rel-cluster-f-as-bowup}).
\end{theorem}
\begin{proof}
  If follows from \cref{prop:foctoring-morphisms-g'-gE} and \cref{rmk:factoring-through-X'}.
\end{proof}

\begin{corollary}\label{coro:closedness-of-X'-in-X1}
  The scheme \(X'\) is a closed subscheme of \(X_1^{\cns}\) and \(X_1\).
\end{corollary}

\begin{proposition}\label{prop:factoring-integral-case}
  Let \((W,\theta)\) be a 2-relative cluster family of \(\pi\) with \(W\) integral.
  Then the scheme \(W\) is equal to either \(W'\) or \(W_E\).
\end{proposition}
\begin{proof}
  Since \(W\) and \(B\) are integral, the locally principal closed subscheme \(E_W\) of \(B\times W\) is either an effective Cartier divisor or isomorphic to \(B\times W\).
  So, for the former case the claim follows from \cref{rmk:factoring-through-X'} and for the later it is trivial.
\end{proof}

\begin{theorem}\label{thr:factoring-through-both-X'-XEc}
  Let \((W,\theta)\) be a 2-relative cluster family of \(\pi\).
  The scheme \(W_{\red}\) is a closed subscheme of the schematic union \(W'+W_E\) (see~\cref{not:schematic-unions}) of \(W'\) and \(W_E\).
  In particular, the underlying topological spaces of \(X'+X_E^{\cns}\) and \(X_1^{\cns}\) are homeomorphic.
\end{theorem}
\begin{proof}
  By \cref{prop:factoring-integral-case} every irreducible component of \(W\), with its reduced structure, is a closed subscheme of either \(W'\) or \(W_E\).
\end{proof}

Let \((W,\theta)\) be a 2-relative cluster family of \(\pi\).
The closed subscheme \(W_E\) of \(W\) is a section family of \(E\to B\).
Observe that, once we get the closed subscheme \(W'\) of \(W\), there is another natural (and maybe more intuitive) way to obtain a closed subscheme of \(W\) corresponding to a section family of \(E\to B\).
Namely, as the schematic closure of the open embedding \((W\setminus W')\tohook W\), let us call it \(W_E^{ii}\).
This two constructions are equivalent in some cases, but in general, there is just closed embeddings \(W_{\red}\tohook (W'\cup W_E^{ii})\tohook (W'\cup W_E)\tohook W\).
Let us show it with a couple of examples unrelated to section families.

We consider \(B\) the spectrum of the base field, so the projection \(W\times B\to W\) is just the identity and the scheme \(W_E\) is equal to \(E_W\).
An example where \(W_E=W_E^{ii}\) is when \(W\) is the spectrum of \(A=\field[x,y]/(xy)\) and \(E_W\) is the principal subscheme determined by \(x\in\field[x,y]/(xy)\).
In this case, the closed embedding \(W'\tohook W\) corresponds to the natural homomorphism \(A\to A/(y)\) and both, \(W_E\) and \(W_E^{ii}\), are the spectrum of \(A/(x)\).
But if we collapse the line \((y)\subseteq A\) to a non-reduced point, that is \(A=\field[x,y]/(y^2,xy)\), then \(W_E^{ii}\) is empty whereas the schemes \(W_E\cup W'\) and \(W\) are equal.

\section{Examples}
\label{ssec:examples}
In this section, we recover two classic constructions, the classic blow up (see \cref{prop:blowup-example}) and an example of a small resolution, both as particular cases of the blow up \S family.

We also present an example showing that the blow up \S family may also behave quite different from such classic constructions, namely, the dimension of the ambient scheme may decrease.
\medskip

\paragraph{The classic blow up.}
The following proposition shows the classic blow up as a particular case of the blow up \S family.

\begin{proposition}\label{prop:blowup-example}
  Consider \cref{sit:blow-up-section-family}.
  Assume that there is a closed subscheme \(W\) of \(X\) such that \(Z=W_Y\).
  Assume that the structure morphism \(Y\overset{\beta}{\to} S\) is affine and the \(\mathscr{O}_S\)-module \(\beta_{*}\mathscr{O}_Y\) is locally free.
  Let \(b\from \mathfrak{B}\to X\) be the blow up of \(X\) along \(W\).
  Then, the couple \((\mathfrak{B},b)\) is the blow up \S family of \(\pi\from X_Y\to Y\) along \(W_Y\).
  In particular, when \(\beta=\Id_S\), the blow up \S family agrees with the classic blow up.
\end{proposition}
\begin{proof}
  The only delicate point is whether the pullback of the exceptional divisor in \(\mathfrak{B}\) by the projection \(\mathfrak{B}_Y\to \mathfrak{B}\) is again an effective Cartier divisor.
  But it follows straightforwardly using the following fact.

  By the assumptions on \(\beta\), affine locally it is given by homomorphisms \(A\to B\) such that \(B\) is a free \(A\)-module (see~\cite[\stacks{01LL},\stacks{01C6},\stacks{01S8}]{stacks-project}).
\end{proof}

\paragraph{The dimension may decrease.}\label{prg:dimension-may-decrease}
We show an example of the blow up \S family where an irreducible ambient space breaks down into two irreducible components and the dimension of one of them decrease by one.
\medskip

Consider \(\mathcal{S}=\proj^1_{u,v}\times \proj^2_{x,y,z}\) and \(Z\subseteq \mathcal{S}\) the graph of \([u:v]\in \proj^1\to [u:v:0]\in \proj^2\), that is \(Z=V_+(z,vx-uy)\).
\medskip

By \cref{thr:structure-of-blowup-sf-Cartesian-case}, the blow up \S family of the projection \(\mathcal{S}\to \proj^1\) along \(Z\) is the stratification of \(\proj^2\) by the standard affine chart \(\proj^2\setminus V_+(z)\) and \(V_+(z)\).

\paragraph{Small resolution.} We present an example where the blow up \S family along a natural centre becomes a small resolution.
It indicates the possibility that the blow up \S family would offer a procedure to systematise small resolutions.
\medskip

Let \(\field\) be a field and consider the variety \(\aff^4_{\field}\) parametrising matrices
\[
  M=\left(
    \begin{array}{cc}
      x & y \\
      z & w \\
\end{array} \right)
\] and the closed subvariety \(D\subseteq \aff^4\) where the rank of \(M\) is not maximal, or equivalently where the determinant of \(M\) is zero.
Consider the variety \(\mathcal{S}=\proj^1_{u,v}\times D\) and its incidence subvariety
\[
  Z=\{([\lambda],M)\in \mathcal{S} \mbox{ : }M\lambda^t=0\}.
\]
It is a classic result that the projection \(\mathcal{S}\to D\) restricted to \(Z\) is an small resolution of \(D\).
It turns out that the blow up \S family of the projection \(\mathcal{S}\to \proj^1\) along \(Z\) is isomorphic to \(Z\) and then again an small resolution of \(D\).
\medskip

Observe that, by \cref{thr:structure-of-blowup-sf-Cartesian-case}, the variety \(D\setminus\{0\}\) is an open subvariety of such a blow up \S family.
But we do not retrieve the whole ambient variety from this result.
Instead, we replicate the construction of the blow up \S family in \cref{thr:blow-up-S-family-exists}.
\medskip

First, let us construct the following quasiprojective varieties \(V_n\).
Let \(S\) denote stander graded polynomial ring \(\field[u,v]\) and \(S_n\) its degree \(n\) part.
So, we define \(U_n\subseteq \proj(S_n\times S_n\times S_n)\) as the quasiprojective variety corresponding to triplets of forms with no common roots.
\medskip

The blow up \(\tilde{\mathcal{S}}\) of \(\mathcal{S}\) along \(Z\) may be given globally by the equations \(xa-zb\) and \(ya-wb\) in \(\mathcal{S}\times \proj^1_{a,b}\).

Now, we describe the closed subvariety of the universal section family of \(\tilde{\mathcal{S}}\to\proj^1\) corresponding to ``constfy'' by \(\tilde{\mathcal{S}}\to D\).
Clearly, it is the disjoint union \(X\) for all integers \(n\) of the closed subvarieties \(X_n\) of \(D\times V_n\) determined by the equations on the coefficients given by the identities of polynomials,
\begin{align*}
  xA-zB\equiv 0  \\
  yA-wB\equiv 0
\end{align*} where \([A:B]\in V_n\).
The resulting morphism \(b'\from X\to D\) is for each component \(X_n\) the composition of the closed embedding \(X_n\tohook D\times V_n\) and the projection \(D\times V_n\to D\).
\medskip

It is straightforward to see that given \(((x,y,z,w),[A:B])\in X_n\) either the forms \(A,B\) are constants or \((x,y,z,w)=0\).
That is,
\[
  X= X_0\ {\textstyle\coprod}\ \Bigl(\coprod_{n\ge 1}\{0\}\times V_n\Bigr)
\] where \(X_0\cong Z\).
So, the pullback \((\Id_{\proj^1}\times b')^{-1}(Z)\) is an effective Cartier divisor in \(\proj^1\times X_0\) and the whole \(\proj^1\times X_n\) for all \(n\ge 1\).
Hence the blow up of \(\proj^1\times X\) along the locally principal \((\Id_{\proj^1}\times b')^{-1}(Z)\) is \(\proj^1\times X_0\), and then the blow up \S family of \(\mathcal{S}\to\proj^1\) along \(Z\) is \(b'|_{X_0}\from X_0\to D\).

\bibliographystyle{amsplain}
\bibliography{my}

\end{document}